\documentclass[preprint,12pt]{elsarticle}

\usepackage{graphicx}
\usepackage[cp1251]{inputenc}
\usepackage{amssymb}
\usepackage{amsthm}
\usepackage{cmap}
\usepackage{picture}

\usepackage{amsmath}

\biboptions{sort&compress}

\newcommand{\sysn}{\left\{\begin{array}{rcl}}
\newcommand{\sysk}{\end{array}\right.}

\newtheorem{theorem}{Theorem}[section]
\newtheorem{lemma}[theorem]{Lemma}

\theoremstyle{example}

\newtheorem{proposition}[theorem]{Proposition}

\newtheorem{corollary}[theorem]{Corollary}

\theoremstyle{definition}
\newtheorem{definition}[theorem]{Definition}

\newtheorem{question}[theorem]{Question}

\journal{...}

\begin{document}

\title{Relatively functionally countable subsets of products}

\author{Anton E. Lipin}

\address{Krasovskii Institute of Mathematics and Mechanics, \\ Ural Federal
 University, Yekaterinburg, Russia}

\ead{tony.lipin@yandex.ru}

\begin{abstract} A subset $A$ of a topological space $X$ is called {\it relatively functionally countable} ({\it RFC}) in $X$, if for each continuous function $f : X \to \mathbb{R}$ the set $f[A]$ is countable. We prove that all RFC subsets of a product $\prod\limits_{n\in\omega}X_n$ are countable, assuming that spaces $X_n$ are Tychonoff and all RFC subsets of every $X_n$ are countable.
In particular, in a metrizable space every RFC subset is countable.

The main tool in the proof is the following result:
for every Tychonoff space $X$ and any countable set $Q \subseteq X$ there is a continuous function $f : X^\omega \to \mathbb{R}^2$ such that the restriction of $f$ to $Q^\omega$ is injective.
\end{abstract}


\maketitle 


\section{Introduction}

This paper comes from the following

\begin{question}[A.V.~Osipov]
\label{IQ1}
Is there an uncountable set $A \subseteq [0,1]^\omega$ such that for every continuous function $f : [0,1]^\omega \to \mathbb{R}$ the set $f[A]$ is countable?
\end{question}

The author knows this problem from private communication with A.V.~Osipov.
In 2016 Osipov and A.~Miller proved that the answer is no under $\frak{b}>\omega_1$ or $\mathrm{cov}(\mathrm{meager})>\omega_1$, but they did not publish these results.
Apparently, no one ever raised Question \ref{IQ1} in literature, although one can find some discussion at this problem on mathoverflow \cite{mathoverflow}, where user fedja and T.~Banakh proved the negative answer under $\neg \mathrm{CH}$. In this paper we prove the negative answer in ZFC.

Instead of the term {\it projectively countable} as in \cite{mathoverflow}, let us use the following notion.

\begin{definition}
A subset $A$ of a topological space $X$ is called {\it relatively functionally countable} ({\it RFC}) in $X$, if for every continuous function $f : X \to \mathbb{R}$ the set $f[A]$ is countable.
\end{definition}

Note that being an RFC subset is not the same as being a functionally countable subspace.
For instance, if $A$ is an uncountable discrete and $X$ is one-point compactification of $A$, then $A$ is an RFC subset and is not a functionally countable subspace.
However, every functionally countable subspace is an RFC subset.
A nice criterion of functional countability for perfectly normal spaces can be found in \cite{tkachuk}.

The most complicated result of this paper is (rather technical) Theorem \ref{AT1}. Section \ref{A} is devoted entirely to its proof.
Corollaries of Theorem \ref{AT1} in terms of functional countability and, in particular, the negative answer to Question \ref{IQ1} can be found in Section \ref{B}.

\section{Preliminaries}

We assume the following notation and conventions.

\begin{itemize}

\item We use the term {\it tree} and associated notation as in \cite[Definition III.5.1]{kunen} or \cite[Definition 9.10]{jech} with the following amendments:
\begin{itemize}
\item instead of ``immediate successor'' we say {\it son};
\item the set of all leaves of a tree $T$ is denoted by $L(T)$;
\item we write $T \preceq P$ to say that $T$ is a subtree of $P$.
\end{itemize}

\item if a function $f: X \to \mathbb{R}^2$ is bounded, then we denote $||f||$ the uniform norm of $f$ in any norm of the plane $\mathbb{R}^2$.

\item the restriction of a function $f : X \to Y$ to a set $A \subseteq X$ is denoted by $f|_A$.

\item {\it countable} means not greater than countable.

\end{itemize}

The following proposition plays a significant role in our argumentation.

\begin{proposition}
\label{0P1}
Denote $\square = [-1,1]\times[0,1]$ and $$\bot = \{(x,y) \in \square : x=0 \text{ or } y=0\}.$$
Suppose we are given a countable set $Q \subseteq \square \setminus \{(0,0)\}$.
Then there is a continuous function $\varphi : \square \to \bot$ such that (whenever $x \in [-1,1]$ and $a \in Q$):
\begin{enumerate}
\item[(1)] $\varphi(x,0) = (x,0)$;
\item[(2)] $\varphi(x,1) = (0,1)$;
\item[(3)] $\varphi(a) \ne (0,0)$.
\end{enumerate}
\end{proposition}
\begin{proof}
	For an illustration see Figure 1 below.
	
	Let us at first define the restriction $\varphi_R$ of $\varphi$ to the set  $\square_R = \{(x,y) \in \square : x \geq 0\}$. Denote by $\triangle_R$ the triangle $\{(x,y) \in \square_R : 0 \leq x \leq 1-y\}$ and $\bot_R = \{(x,y) \in \bot : x \geq 0\}$. For every $(x_0,y_0) \in \square_R$ denote by $f_R(x_0,y_0)$ the nearest to $(x_0,y_0)$ point, belonging both to $\triangle_R$ and the horizontal line $y=y_0$. In particular, $f_R(p)=p$ whenever $p \in \triangle_R$.
	
	Since $Q$ is countable, there is a point $p_R = (x_R, y_R)$ such that $x_R > 0$, $y_R > 0$ and the ray $((0,0),p_R)$ contains no points of the set $f_R[Q \cap \square_R]$. Denote by $g_R$ the projection of $\triangle_R$ onto $\bot_R$ parallel to the direction $(p_R, (0,0))$. Denote by $\varphi_R$ the composition of $f_R$ and $g_R$.
	
	In a dual way we construct the function $\varphi_L$ on the set $\square_L = \{(x,y) \in \square : x \leq 0\}$. Finally, define $\varphi(x,y)$ to be $\varphi_R(x,y)$ for $x \geq 0$ and $\varphi_L(x,y)$ for $x < 0$.
\end{proof}

\begin{picture}(360, 70)
	\put(50,20){\line(1,0){100}}
	\put(50,20){\line(0,1){50}}
	\put(150,70){\line(-1,0){100}}
	\put(150,70){\line(0,-1){50}}
	
	\put(50,20){\line(1,1){50}}
	\put(150,20){\line(-1,1){50}}
	
	\put(50,30){\vector(1,0){10}}
	\put(50,40){\vector(1,0){20}}
	\put(50,50){\vector(1,0){30}}
	\put(50,60){\vector(1,0){40}}
	\put(50,70){\vector(1,0){50}}
	
	\put(150,30){\vector(-1,0){10}}
	\put(150,40){\vector(-1,0){20}}
	\put(150,50){\vector(-1,0){30}}
	\put(150,60){\vector(-1,0){40}}
	\put(150,70){\vector(-1,0){50}}
	
	\put(75, 5){$f_L$ and $f_R$}
	
	\put(210,20){\line(1,0){100}}
	\put(210,20){\line(1,1){50}}
	\put(310,20){\line(-1,1){50}}
	
	\put(260,20){\line(0,1){50}}
	
	\put(300,50){\circle*{5}}
	\put(303,40){$p_R$}
	\put(260,20){\line(4,3){40}}
	\put(300,30){\vector(-4,-3){13}}
	\put(290,40){\vector(-4,-3){26}}
	\put(280,50){\vector(-4,-3){20}}
	\put(270,60){\vector(-4,-3){10}}
	
	\put(220,40){\circle*{5}}
	\put(207,32){$p_L$}
	\put(260,20){\line(-2,1){40}}
	\put(220,30){\vector(2,-1){20}}
	\put(230,40){\vector(2,-1){30}}
	\put(240,50){\vector(2,-1){20}}
	\put(250,60){\vector(2,-1){10}}
	
	\put(235, 5){$g_L$ and $g_R$}
\end{picture}
\begin{center}
	Figure 1
\end{center}

Also let us note a trivial fact on RFC sets.

\begin{proposition}
\label{0P2}
If a subset $A$ of a topological space $X$ is RFC, then for every natural $n$ and any continuous $f : X \to \mathbb{R}^n$ the set $f[A]$ is countable.
\end{proposition}

\section{Constructing maps into the plane $\mathbb{R}^2$}
\label{A}

In this section we use functions, whose codomains are sets of some special form in the plane $\mathbb{R}^2$. To define this special form we need the following two definitions.

\begin{definition}
We say that a triple $(p,I,U)$ is a {\it sprig}, if the following three conditions are satisfied:
\begin{enumerate}
\item[(a)] $p$ is a point in $\mathbb{R}^2$;
\item[(b)] $I$ is a line segment in $\mathbb{R}^2$ and $p$ is an endpoint for $I$;
\item[(c)] $U$ is a neighborhood of $I \setminus \{p\}$.
\end{enumerate}
\end{definition}

\begin{definition}
We say that a rooted tree $(T,<)$ of finite height is an {\it oak}, if all its elements are sprigs and every $t = (p,I,U) \in T$ satisfies the following two conditions:
\begin{enumerate}
\item[(A)] if $(q,J,V)$ is a son of $t$, then $q \in I \setminus \{p\}$, $\overline{V} \subseteq U$ and $\overline{V} \cap I = \{q\}$ \\ (in particular, $J \cap I = \{q\}$);
\item[(B)] if $(q_1,J_1,V_1)$ and $(q_2,J_2,V_2)$ are distinct sons of $t$, then $\overline{V_1} \cap \overline{V_2} = \emptyset$.
\end{enumerate}

If $(T,<)$ is an oak and $T = \{(p_t,I_t,U_t) : t \in T\}$, then we denote:
\begin{itemize}
\item $Y(T) = \bigcup\limits_{t \in T} I_t$;
\item $K(T) = \{p_t : t \in T\}$ (from \underline{K}nots);
\item $C(T) = \bigcup\limits_{t \in L(T)} (I_t \setminus \{p_t\})$ (from \underline{C}rown);
\item $R(T) = Y(T) \setminus C(T)$.
\end{itemize}
\end{definition}

Note that the order $<$ of an oak $(T,<)$ is determined by the family $T$, so we lose no information writing $T$ instead of $(T,<)$.

\medskip

\begin{picture}(360, 80)
	\put(120,40){\circle*{7}}
	\put(120,40){\line(1,0){100}}
	\put(180,40){\oval(120,80)}
	
	\put(150,40){\circle*{5}}
	\put(150,40){\line(1,1){30}}
	\qbezier(150,40)(185,40)(185,75)
	\qbezier(150,40)(150,75)(185,75)
	
	\put(190,40){\circle*{5}}
	\put(190,40){\line(0,-1){30}}
	\put(190,22){\circle{36}}
	
	\put(190,20){\circle*{3}}
	\put(190,20){\line(-1,0){10}}
	\put(182,20){\circle{16}}
\end{picture}
\begin{center}
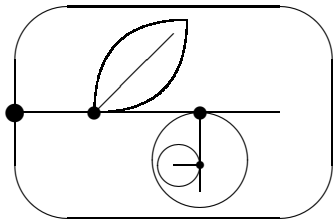
	Figure 2. An oak with four sprigs.
\end{center}

We will work with functions of the form $X \to Y(T)$, where $X$ is a space and $T$ is an oak.
We need these functions as approximations for the function we construct in Theorem \ref{AT1} below.

\begin{definition}
Suppose $T \preceq P$ are oaks and $X$ is a space.
We say that a function $g : X \to Y(P)$ is a {\it $(T,P)$-evolution} of a function $f : X \to Y(T)$, if for all $a \in X$ we have:
\begin{enumerate}
	\item[(1)] if $f(a) \in R(T)$, then $g(a) = f(a)$;
	
	\item[(2)] if $f(a) \in I$ for $(p,I,U) = t \in L(T)$, then either $g(a) \in I$ or $g(a) \in J$ for some son $(q,J,V)$ of $t$ in $P$.
\end{enumerate}
\end{definition}

\begin{definition}
Suppose $T$ is an oak, $X$ is a space and $Q \subseteq X$. We say that a continuous function $f : X \to Y(T)$ is a {\it $(T,X,Q)$-lifting}, if $f[Q] \subseteq C(T)$.
\end{definition}

\begin{definition}
Suppose $T$ is an oak, $X$ is a space, $Q \subseteq X$ and $\mathcal{W}$ is a cover of the set $Q$.
We say that a $(T,X,Q)$-lifting $f$ is {\it $(T,X,Q,\mathcal{W})$-splitting},
if for every leaf $(p,I,U) \in L(T)$ either $f^{-1}[I] \cap Q = \emptyset$ or
there is $W \in \mathcal{W}$ such that $f^{-1}[I] \cap Q \subseteq W$.
\end{definition}

\begin{lemma}
\label{AL1}
Suppose $T$ is an oak, $X$ is a Tychonoff space, a set $Q \subseteq X$ is countable, $f$ is a $(T,X,Q)$-lifting,
$\mathcal{W}$ is an open cover of the set $Q$ and $\varepsilon>0$.
Then there is an oak $P \succeq T$ and a $(P,X,Q,\mathcal{W})$-splitting $(T,P)$-evolution $g$ of the function $f$ such that $||f - g||<\varepsilon$.
\end{lemma}

\begin{proof}
Fix any indexing $Q = \{a_n : n \in \omega\}$.
Denote $f_0 = f$ and $T_0 = T$.
Now let us assume that for some $n \in \omega$ we have an oak $T_n \succeq T$ and a continuous function $f_n : X \to Y(T_n)$ with the following properties:
\begin{enumerate}
\item[(A1$\empty_n$)] the set $T_n \setminus T$ is finite;
\item[(A2$\empty_n$)] $f_n[Q] \cap (R(T) \cup K(T_n)) = \emptyset$;
\item[(A3$\empty_n$)] for every $(p,I,U) \in T_n \setminus T$ there is $W \in \mathcal{W}$ such that $f_n^{-1}[I] \subseteq W$.
\end{enumerate}

Note that for $n = 0$ all these conditions are satisfied. In particular, (A2$\empty_0$) follows from the assumption that the function $f$ is $(T,X,Q)$-lifting.

Let us take $t_n = (p_n,I_n,U_n) \in T_n$ such that $f_n(a_n) \in I_n$.
If $t_n \notin T$, define $T_{n+1} = T_n$ and $f_{n+1} = f_n$. In this case all the recursion assumptions are clearly satisfied.

Suppose $t_n \in T$. By (A2$\empty_n$) we have $t_n \in L(T)$. Note that by (A1$\empty_n$) $t_n$ has only finite set of sons in $T_n$.

Choose an open subinterval $H_n \subseteq I_n$ which contains the point $f_n(a_n)$ and is small enough to satisfy the following two conditions:
\begin{enumerate}
\item[(H1)] the length of $H_n$ is less than $\varepsilon / 2^{n+2}$;
\item[(H2)] $\overline{H_n} \cap K(T_n) = \emptyset$.
\end{enumerate}

Take any sprig $s_n=(q_n, J_n, V_n)$ such that
\begin{enumerate}
\item[(S1)] $q_n \in H_n \setminus f_n[Q]$; 
\item[(S2)] $T_{n+1} = T_n \cup \{s_n\}$ is an oak;
\item[(S3)] the length of $J_n$ is less than $\varepsilon / 2^{n+2}$.
\end{enumerate}

Suppose $a_n \in W_n \in \mathcal{W}$. Denote $O_n = W_n \cap f_n^{-1}[H_n]$; note that it is an open set.
Choose any continuous function $h_n : X \to [0,1]$ such that $h_n(a_n) = 1$ and $h_n[X \setminus O_n] = \{0\}$.
Construct $\Gamma_n = \{(f_n(a),h_n(a)) : a \in X\}$; endow $\Gamma_n$ with a topology of a subspace of $Y(T_n) \times [0,1]$.

Now let us take any continuous function $\varphi_n : \overline{H_n} \times [0,1] \to \overline{H_n} \cup J_n$ with the properties as in Proposition \ref{0P1}, namely that, whenever $x \in \overline{H_n}$ and $a \in Q \cap O_n$:
\begin{enumerate}
\item[(F1)] $\varphi_n(x,0)=x$;
\item[(F2)] $\varphi_n(x, 1) \in J_n \setminus \{q_n\}$;
\item[(F3)] $\varphi_n(f_n(a), h_n(a)) \ne q_n$.
\end{enumerate}

Define $\psi_n : \Gamma_n \to Y(T_{n+1})$ as follows:
$$
\psi_n(x,y) = 
\begin{cases}
	\varphi_n(x,y), \; x \in \overline{H_n}; \\
	x, \; x \notin \overline{H_n}.
\end{cases}$$

The function $\psi_n$ is continuous, because for a point $(x,y) \in \Gamma_n$, where $x$ is from the border of $H_n$, we have $y = 0$, so $\varphi_n(x,y) = x$.

Define the function $f_{n+1} : X \to Y(T_{n+1})$ in the following way:
$$f_{n+1}(a) = \psi_n(f_n(a),h_n(a)).$$

In particular, $f_{n+1}(a) = f_n(a)$ for $a \in X \setminus O_n$.
The function $f_{n+1}$ is continuous as a composition of continuous functions.

Let us show that $T_{n+1}$ and $f_{n+1}$ satisfy (A1$\empty_{n+1}$--A3$\empty_{n+1}$).

(A1$\empty_{n+1}$) follows from (A1$\empty_n$) and the fact that $T_{n+1} \setminus T_n = \{s_n\}$.

(A2$\empty_{n+1}$). Let us take any $a \in Q$ and show that $f_{n+1}(a)$ does not belong to $E = R(T) \cup K(T_{n+1}) = R(T) \cup K(T_n) \cup \{q_n\}$.
There are two cases:
\begin{itemize}
	\item[1.] If $a \in X \setminus O_n$, then $f_{n+1}(a) = f_n(a)$. Here we refer to (A2$\empty_n$) and (S1).
	
	\item[2.] If $a \in O_n$, then $f_{n+1}(a) \in \overline{H_n} \cup J_n$.
	The only common point this set have with $E$ is $q_n$. By (F3) and definition of $f_{n+1}$ we have $f_{n+1}(a) \ne q_n$.
\end{itemize}

(A3$\empty_{n+1}$). For the sprig $s_n=(q_n, J_n, V_n)$ we have $f_{n+1}^{-1}[J_n] \subseteq W_n$. For all other sprigs in $T_{n+1} \setminus T$ we refer to (A3$\empty_n$).

It follows from (H1), (S3) and definition of $f_{n+1}$ that $||f_{n+1} - f_n|| < \varepsilon / 2^{n+1}$.
Thus, the sequence $(f_n)_{n \in \omega}$ uniformly converges to some continuous function $g : X \to \mathbb{R}^2$ such that $||f - g||<\varepsilon$.

The tree $P = \bigcup\limits_{n \in \omega} T_n \succeq T$ is an oak. Note that all the sprigs $s_n$ are leaves of $P$ and height of $P$ is not greater than the height of $T$ plus 1.

Let us show that $g[X] \subseteq Y(P)$ and that $g$ is $(T,P)$-evolution of $f$. Take any $a \in X$. We have $f(a) \in I$ for some $(p,I,U) \in T$. There are two cases:
\begin{itemize}
	\item[I.] $f_n(a) \in I$ for all $n \in \omega$. In this case $g(a) \in I \subseteq Y(P)$ as well.
	
	\item[II.] There is $n \in \omega$ such that $f_n(a) \in I$ and $f_{n+1}(a) \notin I$. This is only possible if $I = I_n$ and $f_{n+1}(a) \in J_n$. In this case $f_m(a) = f_{n+1}(a)$ for all $m > n$, hence $g(a) = f_{n+1}(a) \in J_n$.
\end{itemize}

Moreover, if $f(a) \in R(T)$, then $f_n(a) = f(a)$ for all $n \in \omega$ and so $g(a) = f(a)$. Thus, $g$ is a $(T,P)$-evolution of $f$.

Finally, let us show that $g$ is $(P,X,Q,\mathcal{W})$-splitting.
First we have to prove that $g$ is $(P,X,Q)$-lifting, i.e. that $g[Q] \subseteq C(P)$.
Take any $a_n \in Q$. Depending on the case, we have either $t_n \notin T$, so $t_n$ is a leaf $s_k$ for some $k < n$ and $g(a_n) = f_{k+1}(a_n) \in J_k \setminus \{q_k\} \subseteq C(P)$, or $t_n \in T$, so $g(a_n) = f_{n+1}(a_n) \in J_n \setminus \{q_n\} \subseteq C(P)$.

Now take any leaf $t = (p,I,U) \in L(P)$. There are two cases:
\begin{itemize}
	\item[i.] $t = s_n$ for some $n \in \omega$. Here we have $g^{-1}[I \setminus \{q_n\}] = f_{n+1}^{-1}[I \setminus \{q_n\}]$ and $f_{n+1}^{-1}[I] \subseteq W_n$. Moreover, $g^{-1}(q_n) \cap Q = \emptyset$, because $g$ is $(P,X,Q)$-lifting. Hence, $g^{-1}[I] \cap Q \subseteq W_n$.
	
	\item[ii.] $t$ is a leaf of $T$ and $t$ has no sons in $P$. Hence, $g^{-1}[I] \cap Q = \emptyset$.
\end{itemize}
\end{proof}

\begin{proposition}
\label{AP1}
	Suppose $X$ is a $T_1$-space and a set $Q \subseteq X$ is countable. Then it is possible to choose for every $n \in \mathbb{N}$ an open cover $\mathcal{W}_n$ of the set $Q^n$ in the space $X^n$ in such a way that for any distinct points $a,b \in Q^\omega$ there is $n$ such that no element of $\mathcal{W}_n$ contains both $a|_n$ and $b|_n$.
\end{proposition}
\begin{proof}
Choose any indexing $Q = \{x_n : n \in \omega\}$ and denote $D_n = \{x_k : k < n\}$. The set $D_n$ is finite, so there is an open cover $\mathcal{U}_n$ of $Q$ in $X$ such that $|U \cap D_n| \leq 1$ whenever $U \in \mathcal{U}_n$.
Define $$\mathcal{W}_n = \left\{\prod\limits_{k<n} U_k : U_k \in \mathcal{U}_n\right\}.$$

Clearly, $\mathcal{W}_n$ is an open cover of $Q^n$ in $X^n$. Now, if $a$ and $b$ are distinct points of $Q^\omega$, then there is $k \in \omega$ such that $a(k) = x_i \ne x_j = b(k)$. Denote $n = \max\{k, i, j\}+1$. We have $a(k),b(k) \in D_n$, so no element of $\mathcal{U}_n$ contains both $a(k)$ and $b(k)$, hence no element of $\mathcal{W}_n$ contains both $a|_n$ and $b|_n$.
\end{proof}

\begin{theorem}
\label{AT1}
Suppose $X$ is a Tychonoff space and a set $Q \subseteq X$ is countable. Then there is a continuous function $f : X^\omega \to \mathbb{R}^2$ such that its restriction to $Q^\omega$ is injective.
\end{theorem}
\begin{proof}
Choose families $\mathcal{W}_n$ with the properties from Proposition \ref{AP1}.
Take a sprig $t=(p,I,U)$, where $p=(0,0)$, $I = [0,1] \times \{0\}$ and $U = \mathbb{R}^2$.
Denote $T_1 = \{t\}$ and denote $f_1(x) = (1,0)$ for all $x \in X$.

Now suppose that for some $n \in \mathbb{N}$ we have an oak $T_n$ and a $(T_n,X^n,Q^n)$-lifting $f_n : X^n \to Y(T_n)$.

By Lemma \ref{AL1} there is an oak $T_{n+1} \succeq T_n$ and a $(T_{n+1},X^n,Q^n,\mathcal{W}_n)$-splitting $(T_n,T_{n+1})$-evolution $g_n$ of the function $f_n$ such that $||f_n - g_n||<1/2^n$. Define the function $f_{n+1} : X^{n+1} \to Y(T_{n+1})$ in such a way that $f_{n+1}(x_1,\ldots,x_{n+1}) = g_n(x_1,\ldots,x_n)$,
so the pair $f_{n+1}, T_{n+1}$ satisfies the recursion assumption.

Now for all $n \in \mathbb{N}$ define the functions $h_n : X^\omega \to Y(T_n)$ as follows: $h_n(x_1,x_2,\ldots) = f_n(x_1,\ldots,x_n)$.
We have $||h_{n+1} - h_n|| = ||g_n - f_n||< 1/2^n$, so the sequence $(h_n)_{n \in \mathbb{N}}$ uniformly converges to some continuous function $f : X^\omega \to \mathbb{R}^2$.

Let us show that for any distinct $a,b \in Q^\omega$ we have $f(a) \ne f(b)$.
Take $n \in \mathbb{N}$ such that no element of $\mathcal{W}_n$ contains both $a|_n$ and $b|_n$.
Choose $t_a = (p_a, I_a, U_a)$ and $t_b = (p_b, I_b, U_b)$ in $T_{n+1}$ such that $g_n(a|_n) \in I_a$ and $g_n(b|_n) \in I_b$.
Since $g_n$ is $(T_{n+1},X^n,Q^n,\mathcal{W}_n)$-splitting, $t_a$ and $t_b$ are distinct leaves of $T_{n+1}$, so $\overline{U_a} \cap \overline{U_b} = \emptyset$.

It remains to show that $g(a) \in \overline{U_a}$ and $g(b) \in \overline{U_b}$. We shall prove the first condition, since the proof of the second one is the same.
For every $k \in \mathbb{N}$ take the leaf $t_k = (p_k, I_k, U_k)$ of $T_k$ such that $f_k(a|_k) \in I_k$. Since $f_{k+1}(a|_{k+1}) = g_k(a|_k)$, we conclude, by the choice of $g_k$, that $t_{k+1}$ is a son of $t_k$, so $\overline{U_{k+1}} \subseteq U_k$. It follows that $f_k(a|_k) \in U_{n+1} = U_a$ for all $k > n$, hence $f(a) \in \overline{U_a}$.
\end{proof}

Of course, Theorem \ref{AT1} may be extended to products of distinct spaces.

\begin{corollary}
\label{AC1}
Suppose that we are given a Tychonoff space $X_n$ and a countable set $Q_n \subseteq X_n$ for every $n \in \omega$.
Then there is a continuous function $f : \prod\limits_{n \in \omega} X_n \to \mathbb{R}^2$ such that its restriction to $\prod\limits_{n \in \omega} Q_n$ is injective.
\end{corollary}
\begin{proof}
We can suppose that all $X_n$ are pairwise disjoint subsets of some space $X$.
Applying Theorem \ref{AT1} to $X$ and $Q = \bigcup\limits_{n \in \omega} Q_n$, we obtain this corollary.
\end{proof}

\section{Relatively functionally countable subsets of products}
\label{B}

\begin{theorem}
\label{BT1}
Suppose we are given Tychonoff spaces $X_n$ for all $n \in \omega$ and all RFC subsets of every $X_n$ are countable.
Then all RFC subsets of $\prod\limits_{n \in \omega} X_n$ are countable.
\end{theorem}
\begin{proof}
Suppose a set $A \subseteq \prod\limits_{n \in \omega} X_n = Y$ is RFC.
In particular, for every projection $\pi_n : Y \to X_n$, defined as $\pi_n(x_1,x_2,\ldots) = x_n$, the set $Q_n = \pi_n[A]$ is countable.

By Corollary \ref{AC1} there is a continuous function $f : Y \to \mathbb{R}^2$ such that its restriction to $\prod\limits_{n \in \omega} Q_n$ (and hence to $A$) is injective. By Proposition \ref{0P2} the set $f[A]$ is countable, so $A$ is countable too.
\end{proof}

In particular, the answer to Question \ref{IQ1} is no.

\begin{corollary}
Every RFC subset of a metrizable space is countable.
\end{corollary}
\begin{proof}
Let $X$ be a metrizable space.
By \cite[Theorem 4.4.9]{engelking} we may suppose that $X$ is a subspace of $H^\omega$ for some metrizable hedgehog $H$.
All RFC subsets of $X$ are RFC in $H^\omega$ as well.
By Theorem \ref{BT1}, to prove that all these sets are countable, it suffices to show that every RFC subset of $H$ is countable.
So, let us take an arbitrary uncountable set $A \subseteq H$ and show that $A$ is not RFC.

Choose any $B \subseteq A$ such that $|B|=\omega_1$.
Take a hedgehog $J \subseteq H$ such that $B \subseteq J$ and $J$ has no more than $\omega_1$-many spines.
Fix any retraction $r$ of $H$ onto $J$.
It is well known and easy to prove that there is a continuous injection $f : J \to \mathbb{R}^2$.
The function $g : H \to \mathbb{R}^2$, defined as $g(x) = f(r(x))$, is continuous and the set $g[A]$ is uncountable.
According to Proposition \ref{0P2} the set $A$ is not RFC.
\end{proof}

Clearly, there is no possibility to replace $\omega$ by $\omega_1$ in Theorem \ref{BT1}, since the uncountable functionally countable space $\omega_1$ can be embedded into $X^{\omega_1}$ whenever $X$ is Hausdorff and $|X|>1$.

\begin{question}[$\neg CH$]
Is there a set $A \subseteq [0,1]^\omega$ such that $|A|=\frak{c}$ and for every continuous function $f : [0,1]^\omega \to \mathbb{R}$ we have $|f[A]| < \frak{c}$?
\end{question}

More generally,

\begin{question}[$\neg CH$]
Is there a metrizable space $X$ and its subset $A$ such that for every continuous $f : X \to \mathbb{R}$ we have $|f[A]|<|A| \leq \frak{c}$?
\end{question}

\section{Why $\mathbb{R}^2$ or a few words on maps $\mathbb{R}^\omega \to \mathbb{R}$}
\label{C}

One may ask a natural question whether $\mathbb{R}^2$ in Theorem \ref{AT1} can be replaced by $\mathbb{R}$.
Let us show that the answer is negative.

\begin{theorem}
	For any continuous function $f : \mathbb{R}^\omega \to \mathbb{R}$ the restriction of $f$ to $\mathbb{Q}^\omega$ is not injective.
\end{theorem}
\begin{proof}
	For all $x_0,\ldots,x_{n-1} \in \mathbb{Q}$ let us denote $$H(x_0,\ldots,x_{n-1}) = \{p \in \mathbb{R}^\omega : p|_n = (x_0,\ldots,x_{n-1})\}.$$
	In particular, for $n=0$ we define $H() = \mathbb{R}^\omega$.
	Every set $f[H(x_0,\ldots,x_{n-1})]$ is connected as an image of a connected set.
	If some of these sets is a singleton, then $f|_{\mathbb{Q}^\omega}$ is not injective, so we may suppose that all these sets are nontrivial intervals.
	
	Now let us construct two distinct sequences $p = (a_n)_{n \in \omega}$ and $q = (b_n)_{n \in \omega}$ of rational numbers with the property that for all $n \in \omega$ the set $$L_n = f[H(a_0, \ldots, a_{n-1})] \cap f[H(b_0, \ldots, b_{n-1})]$$ is a nontrivial interval.
	
	This property clearly holds for $n = 0$. Now suppose we are given $a_0, \ldots, a_{n-1}$ and $b_0, \ldots, b_{n-1}$ and let us show that it is possible to choose $a_n \ne b_n$.
	
	First, the set $\bigcup_{x \in \mathbb{Q}} H(a_0, \ldots, a_{n-1}, x)$ is dense in $H(a_0, \ldots, a_{n-1})$, so its image is dense in $f[H(a_0, \ldots, a_{n-1})]$. Consequently, there is $a_n \in \mathbb{Q}$ such that the interval $f[H(a_0, \ldots, a_n)]$ contains some interior point of $L_n$, so the interval $$M_n = f[H(a_0, \ldots, a_n)] \cap L_n = f[H(a_0, \ldots, a_n)] \cap f[H(b_0, \ldots, b_{n-1})]$$ is nontrivial.
	
	Similarly, the set $\bigcup_{x \in \mathbb{Q} \setminus \{a_n\}} H(b_0, \ldots, b_{n-1}, x)$ is dense in $H(b_0, \ldots, b_{n-1})$, so there is $b_n \in \mathbb{Q} \setminus \{a_n\}$ such that the interval $f[H(b_0, \ldots, b_n)]$ contains some interior point of $M_n$, hence the interval $M_n \cap f[H(b_0, \ldots, b_n)] = L_{n+1}$ is nontrivial.
	
	Sequences $p \ne q$ are constructed. Let us show that $f(p) = f(q)$. Suppose, on the contrary, that $f(p) \ne f(q)$. It follows that there are neighborhoods $U$ and $V$ of the points $f(p)$ and $f(q)$ in $\mathbb{R}$ such that $U \cap V = \emptyset$. Take $n \in \omega$ such that $H(p|_n) \subseteq f^{-1}[U]$ and $H(q|_n) \subseteq f^{-1}[V]$. We have $f[H(p|_n)] \cap f[H(q|_n)] = \emptyset$, and this is a contradiction.
\end{proof}

\medskip

\noindent {\bf Acknowledgements.} I am grateful to Alexander~V. Osipov for constant attention to this work and to the referee, who made me lay out the reasoning much more accurately.

\bibliographystyle{model1a-num-names}
\bibliography{<your-bib-database>}

\begin{thebibliography}{10}

\bibitem{engelking} R. Engelking, General Topology, Revised and completed edition, Heldermann Verlag Berlin (1989).

\bibitem{jech} T.~Jech, Set Theory: The third millenium edition (3rd ed.). Berlin: Springer-Verlag (2006).

\bibitem{kunen} Kenneth Kunen, Set Theory, Studies in Logic: Mathematical Logic and Foundations, Vol. 34, College Publications, London, 2011, viii + 401 pp.

\bibitem{tkachuk} V.V.~Tkachuk, A nice subclass of functionally countable spaces, {\it Comment. Math. Univ. Carolin.}, 59,3 (2018) 399--409.

\bibitem{mathoverflow} https://mathoverflow.net/questions/330255/on-projectively-countable-sets-in-the-hilbert-cube
\end{thebibliography}

\end{document}